\documentclass[12pt, reqno]{amsart}
\usepackage{amsmath, amsthm, amscd, amsfonts, amssymb, graphicx, color}
\usepackage[bookmarksnumbered, colorlinks, plainpages]{hyperref}

\input{mathrsfs.sty}

\newtheorem{theorem}{Theorem}[section]
\newtheorem{lemma}[theorem]{Lemma}
\newtheorem{proposition}[theorem]{Proposition}
\newtheorem{corollary}[theorem]{Corollary}
\theoremstyle{definition}

\theoremstyle{remark}
\newtheorem{remark}[theorem]{Remark}
\numberwithin{equation}{section}
\begin{document}

\title [Extensions of interpolation between the arithmetic-geometric]{Extensions of interpolation between the arithmetic-geometric mean inequality for matrices}

\author[M. Bakherad, R. Lashkaripour, M. Hajmohamadi]{M. Bakherad$^1$, R. Lashkaripour$^2$ and  M. Hajmohamadi$^3$}

\address{$^1$$^2$$^3$Department of Mathematics, Faculty of Mathematics, University of Sistan and Baluchestan, Zahedan, I.R.Iran.}

\email{$^1$mojtaba.bakherad@yahoo.com; bakherad@member.ams.org}
\email{$^{2}$lashkari@hamoon.usb.ac.ir}
\email{$^{2}$monire.hajmohamadi@yahoo.com}

\begin{abstract}
In this paper, we present some extensions of interpolation between the arithmetic-geometric means inequality.  Among other inequalities, it is
shown that if $A, B, X$ are $n\times n$ matrices, then
\begin{align*}
\|AXB^*\|^2\leq\|f_1(A^*A)Xg_1(B^*B)\|\,\|f_2(A^*A)Xg_2(B^*B)\|,
\end{align*}
where $f_1,f_2,g_1,g_2$ are non-negative continues functions such that $f_1(t)f_2(t)=t$ and $g_1(t)g_2(t)=t\,\,(t\geq0)$. We also obtain the inequality
\begin{align*}
\left|\left|\left|AB^*\right|\right|\right|^2\nonumber&\leq
\left|\left|\left|p(A^*A)^{\frac{m}{p}}+ (1-p)(B^*B)^{\frac{s}{1-p}}\right|\right|\right|\,\left|\left|\left|(1-p)(A^*A)^{\frac{n}{1-p}}+ p(B^*B)^{\frac{t}{p}}\right|\right|\right|,
\end{align*}
in which  $m,n,s,t$ are real numbers such that $m+n=s+t=1$,  $|||\cdot|||$ is an arbitrary unitarily invariant norm and $p\in[0,1]$.
\end{abstract} \maketitle
\section{Introduction and preliminaries}

\noindent Let $\mathcal{M}_n$ be the $C^*$-algebra of all
$n\times n$ complex matrices and  $\langle\,\cdot\,,\cdot\,\rangle$ be the standard scalar
product in $\mathbb{C}^n$ with the identity $I$.
The Gelfand map $f(t)\mapsto f(A)$ is an
isometrical $*$-isomorphism between the $C^*$-algebra
$C({\rm sp}(A))$ of continuous functions on the spectrum ${\rm sp}(A)$
of a Hermitian matrix $A$ and the $C^*$-algebra generated by $A$ and $I$.

A norm $|||\cdot|||$ on $\mathcal{M}_n$ is said to be unitarily invariant norm if $|||UAV|||=|||A|||$, for all unitary matrices $U$ and $V$. For $A\in\mathcal{M}_n$, let $s_1(A) \geq s_2(A) \geq \cdots  \geq s_n(A)$ denote the singular
values of $A$, i.e. the eigenvalues of the positive semidefinite matrix $|A| = (A^*A)^\frac{1}{2}$ arranged
in a decreasing order with their multiplicities counted. Note that $s_j(A)=s_j(A^*)=s_j(|A|)\,\,(1\leq j \leq n)$ and $\|A\|=s_1(A)$. The Ky Fan norm of a matrix $A$ is defined as $\|A\|_{(k)}=\sum_{j=1}^k s_j(A)\,\,(1\leq k\leq n)$.
The Fan dominance theorem asserts that $\|A\|_{
(k)} \leq\|B\|_{(k)}$
 for $k = 1, 2, \cdots, n$
if and only if $|||A||| \leq |||B|||$ for every unitarily invariant norm (see \cite[p.93]{bha3}). The Hilbert-Schmidt norm is defined
by $\|A\|_2=\left(\sum_{j=1}^ns_j^2(A)\right)^{1/2}$, where is  unitarily invariant. \\
The classical Cauchy-Schwarz inequality for $a_j\geq0$, $b_j\geq0\,\,(1\leq j\leq n)$ states that
 \begin{align*}
\left(\sum_{j=1}^n a_jb_j  \right)^2\leq\left(\sum_{j=1}^n a_j^2 \right)\left(\sum_{j=1}^n b_j^2  \right)
 \end{align*}
 with equality if and only if  $(a_1,\cdots, a_n)$  and $(b_1,\cdots, b_n)$ are proportional \cite{pec}.
  Bhatia and Davis gave a matrix Cauchy-Schwarz inequality as follows
 \begin{align}\label{a24}
||| AXB^*|||^2\leq |||A^*AX|||\,|||XB^*B|||
 \end{align}
 where $A, B, X\in\mathcal{M}_n$. (For further information about the Cauchy-Schwarz inequality, see \cite{bakh0, bakh2, bakh4} and
references therein.) Recently, Kittaneh  et al. \cite{kit1} extended inequality \eqref{a24} to the form
 \begin{align}\label{kiteq1}
||| AXB^*|||^2\leq |||(A^*A)^pX(B^*B)^{1-p}|||\,|||(A^*A)^{1-p}X(B^*B)^{p}|||,
 \end{align}
where $A, B, X\in\mathcal{M}_n$ and $p\in[0,1]$.
Audenaert \cite{ader}  proved that for all $A, B\in\mathcal{M}_n$ and all $p \in [0, 1]$, we have
 \begin{align}\label{ader}
||| AB^*|||^2\leq |||pA^*A+(1-p)B^*B|||\,|||(1-p)A^*A+pB^*B|||.
 \end{align}
 In \cite{zo}, the authors  generalized inequality \eqref{ader}  for all $A, B, X\in\mathcal{M}_n$ and all $p \in [0, 1]$ to the form
 \begin{align}\label{newader}
||| AXB^*|||^2\leq |||pA^*AX+(1-p)XB^*B|||\,|||(1-p)A^*AX+pXB^*B|||.
 \end{align}
 Inequality \eqref{newader} interpolates between the arithmetic-geometric mean inequality. In \cite{kit1}, the authors  showed a refinement of inequality \eqref{newader} for the  Hilbert-Schmidt norm as follows
{\footnotesize\begin{align}\label{kiteq2}
\left\|AXB^*\right\|_2^2&\leq
\left(\left\|p A^*A X+(1-p) XB^*B\right\|_2^2\,-r^2\left\|A^*AX-XB^*B\right\|_2^2\right)\nonumber\\&
\quad\times\left(\left\|(1-p) A^*A X+p XB^*B\right\|_2^2\,-r^2\left\|A^*AX-XB^*B\right\|_2^2\right),
\end{align}}
in which $A, B, X\in\mathcal{M}_n$, $p\in[0,1]$ and $r=\min\{p,1-p\}$.
The Young inequality for every unitarily invariant norm states that $|||A^pB^{1-p}|||\leq|||pA+(1-p)B|||,$ where $A, B$ are positive definite matrices and $p\in[0,1]$ (see \cite{ando} and also \cite{bakh}). Kosaki \cite{kos}
extended the last inequality for the Hilbert-Schmidt norm as follows
\begin{eqnarray}\label{momo2}
\|A^pXB^{1-p}\|_2\leq\|pAX+(1-p)XB\|_2,
\end{eqnarray}
where $A, B$ are positive definite matrices,  $X$ is any matrix and $p\in[0,1]$. In \cite{hirz}, the authors  considered as a refined matrix
Young inequality for the Hilbert-Schmidt norm
\begin{eqnarray}\label{momo}
\|A^pXB^{1-p}\|_2^2+r^2\|AX-XB\|_2^2\leq\|pAX+(1-p)XB\|_2^2,
\end{eqnarray}
in which $A, B$ are positive semidefinite matrices, $X\in\mathcal{M}_n$, $p\in[0,1]$ and $r=\min\{p,1-p\}$.\\
Based on the refined Young inequality \eqref{momo}, Zhao and Wu \cite{zh} proved that
\begin{align}\label{mm}
\|A^pXB^{1-p}\|_2^2+r_0\|A^{\frac{1}{2}}XB^{\frac{1}{2}}-AX\|_2^2+(1-p)^{2}\|AX-XB\|_2^2\leq\|pAX+(1-p)XB\|_2^2,
\end{align}
for $0<p\leq\frac{1}{2}$ and
\begin{align*}
\|A^pXB^{1-p}\|_2^2+r_0\|A^{\frac{1}{2}}XB^{\frac{1}{2}}-XB\|_2^2+p^{2}\|AX-XB\|_2^2\leq\|pAX+(1-p)XB\|_2^2,
\end{align*}
for $\frac{1}{2}<p<1$ such that $r=\min\{p, 1-p\}$ and $r_0=\min\{2r, 1-2r\}$.

In this paper, we obtain some operator and unitarily invariant norms inequalities. Among other results,  we obtain a refinement of inequality \eqref{kiteq2} and  we also extend inequalities \eqref{kiteq1}, \eqref{ader} and \eqref{kiteq2} for
the function $f(t)=t^p\,\,(p\in\mathbb{R})$.

\section{Main results}
\bigskip In this section, by using some ideas of \cite{kit1} we extend the Audenaert results for the operator norm.\\
\begin{theorem}\label{3}
Let $A, B, X\in\mathcal{M}_n$  and $f_1,f_2,g_1,g_2$ be non-negative continues functions such that $f_1(t)f_2(t)=t$ and $g_1(t)g_2(t)=t\,\,(t\geq0)$. Then
\begin{align}\label{eq1}
\|AXB^*\|^2\leq\|f_1(A^*A)Xg_1(B^*B)\|\,\|f_2(A^*A)Xg_2(B^*B)\|.
\end{align}
\end{theorem}
\begin{proof}
It follows from
\begin{eqnarray*}
\|AXB^*\|^2&=&\|BX^*A^*AXB^*\|\\&=&s_1\big(BX^*A^*AXB^*\big)\\&=&\lambda_{\max}\big(BX^*A^*AXB^*\big)\quad(\textrm{since}\,BX^*A^*AXB^*\, \textrm{is positive semidefinite}) \\&=&
\lambda_{\max}\big(X^*A^*AXB^*B\big)\\&=&\lambda_{\max}\big(X^*f_1(A^*A)f_2(A^*A)Xg_2(B^*B)g_1(B^*B)\big)
\\&=&\lambda_{\max}\big(g_1(B^*B)X^*f_1(A^*A)f_2(A^*A)Xg_2(B^*B)\big)
\\&\leq&\left\|g_1(B^*B)X^*f_1(A^*A)f_2(A^*A)Xg_2(B^*B)\right\|\\&\leq&\left\|g_1(B^*B)X^*f_1(A^*A)\right\|\,\left\|f_2(A^*A)Xg_2(B^*B)\right\|\\&=&
\left\|f_1(A^*A)Xg_1(B^*B)\right\|\,\left\|f_2(A^*A)Xg_2(B^*B)\right\|
\end{eqnarray*}
that we get the desired result.
\end{proof}
\begin{corollary}
If  $A, B, X\in\mathcal{M}_n$  and $m,n,s,t$ are real numbers such that $m+n=s+t=1$, then
\begin{align}\label{41}
\|AXB^*\|^2\leq\|(A^*A)^mX(B^*B)^s\|\,\|(A^*A)^nX(B^*B)^t\|.
\end{align}
\end{corollary}
In the next results, we show some  generalizations of  inequality \eqref{ader} for the operator norm.
\begin{corollary}\label{lolo1}
Let $A, B\in\mathcal{M}_n$ and let  $f_1,f_2,g_1,g_2$ be non-negative continues functions such that $f_1(t)f_2(t)=t$ and $g_1(t)g_2(t)=t\,\,(t\geq0)$. Then
{\begin{eqnarray*}
\|AB^*\|^2\leq
\left\|pf_1(A^*A)^\frac{1}{p}+(1-p)g_1(B^*B)^\frac{1}{1-p} \right\|\left\|(1-p) f_2(A^*A)^\frac{1}{1-p}+pg_2(B^*B)^\frac{1}{p}\right\|,
\end{eqnarray*}}
where $p\in[0,1]$.
\end{corollary}
\begin{proof}
Using Theorem  \ref{3}  for $X=I$ we have
\begin{eqnarray*}
\|AB^*\|^2&\leq&
\left\|f_1(A^*A) g_1(B^*B)\right\|\left\|f_2(A^*A) g_2(B^*B)\right\|\\&=&
\left\|\left(f_1(A^*A)^\frac{1}{p}\right)^p\left( g_1(B^*B)^\frac{1}{1-p}\right)^{1-p}\right\|\,
\left\|\left(f_2(A^*A)^\frac{1}{1-p}\right)^{1-p}\left( g_2(B^*B)^\frac{1}{p}\right)^p\right\|
\qquad\qquad\qquad\qquad\qquad\qquad(\textrm{by Theorem \ref{3} )}\\&\leq&
\left\|pf_1(A^*A)^\frac{1}{p}+(1-p)g_1(B^*B)^\frac{1}{1-p} \right\|\left\|(1-p) f_2(A^*A)^\frac{1}{1-p}+pg_2(B^*B)^\frac{1}{p}\right\|\\&&
\qquad\qquad\qquad\qquad\qquad\qquad(\textrm{by the Young inequality}).
\end{eqnarray*}
\end{proof}
\begin{corollary}\label{lolo2}
Let $A, B\in\mathcal{M}_n$ and let $f, g$ be non-negative continues functions such that $f(t)g(t)=t^2\,\,(t\geq0)$. Then
{\begin{eqnarray*}
\|AB^*\|^2&\leq&
\|p f(A^*A) +(1-p) g(B^*B)\|^\frac{1}{2}\,\|(1-p) f(A^*A) +p g(B^*B)\|^\frac{1}{2}
\\&&\times
\left\| p g(A^*A)+(1-p)f(B^*B)\right\|^\frac{1}{2}\left\|(1-p)g(A^*A)+p f(B^*B)\right\|^\frac{1}{2},
\end{eqnarray*}}
where $p\in[0,1]$.
\end{corollary}
\begin{proof}
Applying Theorem \ref{3} and the Young inequality  we get
\begin{eqnarray*}
\|AB^*\|^4&\leq&
\left\|f(A^*A)^\frac{1}{2} g(B^*B)^\frac{1}{2}\right\|^2\left\|g(A^*A)^\frac{1}{2} f(B^*B)^\frac{1}{2}\right\|^2\\&&
\qquad\qquad\qquad\qquad\qquad\qquad(\textrm{by Theorem \ref{3} for}\,\sqrt{f}\, \textrm{and}\, \sqrt{g})\\&\leq&
\left\|f(A^*A)^p g(B^*B)^{1-p}\right\|\,\left\|f(A^*A)^{1-p} g(B^*B)^{p}\right\|\\&&\times\left\|g(A^*A)^{p} f(B^*B)^{1-p}\right\|\,\left\|g(A^*A)^{1-p} f(B^*B)^{p}\right\|\\&&
\qquad\qquad\qquad\qquad\qquad\qquad(\textrm{by inequality \eqref{41}})\\&\leq&
\left\| p f(A^*A)+(1-p)g(B^*B)\right\|\left\|(1-p)f(A^*A)+p g(B^*B)\right\|\\&&\times
\left\| p g(A^*A)+(1-p)f(B^*B)\right\|\left\|(1-p)g(A^*A)+p f(B^*B)\right\|\\&&
\qquad\qquad\qquad\qquad\qquad\qquad(\textrm{by the Young inequality}).
\end{eqnarray*}
\end{proof}

\section{Some interpolations for unitarily invariant norms}
In this section, by applying some ideas of \cite{kit1} we generalize some interpolations for an arbitrary unitarily invariant norm.\\
Let $Q_{k,n}$ denote the set of all strictly increasing $k$-tuples chosen from ${1, 2,\cdots, n}$, i.e.
$I \in Q_{k,n}$ if $I = (i_1, i_2, \cdots , i_k)$, where $1\leq i_1<i_2<\cdots< i_k\leq n$. The following lemma
gives some properties of the $k$th antisymmetric tensor
powers of matrices in $\mathcal{M}_n$; see \cite[p.18]{bha3}.
\begin{lemma}\label{1}
Let $A, B\in\mathcal{M}_n$.
Then\\
$(a)\,\, (\wedge^kA)(\wedge^kB)=\wedge^k(AB)$ for $k=1,\cdots,n$.\\
$(b)\,\,(\wedge^kA)^*=\wedge^kA^*$ for $k=1,\cdots,n$.\\
$(c)\,\,(\wedge^kA)^{-1}=\wedge^kA^{-1}$ for $k=1,\cdots,n$.\\
$(d)$\,\,If $s_1, s_2, \cdots , s_n$ are the singular values of $A$, then the singular values of $\wedge^kA$ are
$s_{i_1}, s_{i_2}, \cdots , s_{i_k}$, where $({i_1}, {i_2}, \cdots , {i_k})\in Q_{k,n}$.\\
\end{lemma}
Now, we show inequality \eqref{41} for an arbitrary unitarily invariant norm.
\begin{theorem}\label{main2}
Let $A, B, X\in\mathcal{M}_n$ and $|||\cdot|||$ be an arbitrary unitarily invariant norm. Then
\begin{align}\label{4}
|||AXB^*|||^2\leq|||(A^*A)^mX(B^*B)^s|||\,|||(A^*A)^nX(B^*B)^t|||,
\end{align}
where $m,n,s,t$ are real numbers such that $m+n=s+t=1$. In particular, if $A$, $B$ are positive definite
\begin{eqnarray}\label{hia}
|||A^\frac{1}{2}XB^\frac{1}{2}|||^2\leq|||A^p XB^{1-p}|||\,|||A^{1-p}XB^p|||,
\end{eqnarray}
where $p\in[0,1]$.
\end{theorem}
\begin{proof}
If we replace $A$, $B$ and $X$ by $\wedge^kA$, $\wedge^kB$ and $\wedge^kX$, their $k$th antisymmetric tensor powers in inequality \eqref{eq1} and  apply Lemma \ref{1} , then  we have
\begin{align*}
\left\|\wedge^kAXB^*\right\|^2\leq\left\|\wedge^k(A^*A)^mX(B^*B)^s\right\|\,\left\|\wedge^k(A^*A)^nX(B^*B)^t\right\|
\end{align*}
that is equivalent to
\begin{align*}
s_1^2\left(\wedge^kAXB^*\right)\leq s_1\left(\wedge^k(A^*A)^mX(B^*B)^s\right)\,s_1\left(\wedge^k(A^*A)^nX(B^*B)^t\right).
\end{align*}
Applying Lemma \ref{1}(d), we have
\begin{eqnarray}\label{bobo}
\prod_{j=1}^{k}s_j\left(AXB^*\right)\nonumber&\leq& \prod_{j=1}^{k}s_j^{\frac{1}{2}}\left((A^*A)^mX(B^*B)^s\right)\prod_{j=1}^{k}s_j^{\frac{1}{2}}\left((A^*A)^nX(B^*B)^t\right)
\nonumber\\&\leq& \prod_{j=1}^{k}s_j^{\frac{1}{2}}\left((A^*A)^mX(B^*B)^s\right)s_j^{\frac{1}{2}}\left((A^*A)^nX(B^*B)^t\right),
\end{eqnarray}
where $k=1,\cdots,n$. Inequality \eqref{bobo} implies that
\begin{eqnarray*}
\sum_{j=1}^{k}s_j\left(AXB^*\right)&\leq& \sum_{j=1}^{k}s_j^{\frac{1}{2}}\left((A^*A)^mX(B^*B)^s\right)s_j^{\frac{1}{2}}\left((A^*A)^nX(B^*B)^t\right)
\\&\leq& \left(\sum_{j=1}^{k}s_j\left((A^*A)^mX(B^*B)^s\right)\right)^{\frac{1}{2}}
\left(\sum_{j=1}^{k}s_j\left((A^*A)^nX(B^*B)^t\right)\right)^{\frac{1}{2}}\\&&
\qquad\qquad\qquad\qquad\qquad\qquad\textrm{(by the Cauchy-Schwarz inequality)},
\end{eqnarray*}
where $k=1,\cdots,n$. Hence
\begin{align*}
\left\|AXB^*\right\|_{(k)}^2\leq \|(A^*A)^mX(B^*B)^s\|_{(k)}\|(A^*A)^nX(B^*B)^t\|_{(k)}.
\end{align*}
Now, using the Fan dominace theorem \cite[p.98]{bha3}, we get the desired result.
\end{proof}

Now, by using  inequality \eqref{hia}, Theorem \ref{main2} and a same argument in the proof of Corollaries  \ref{lolo1} and \ref{lolo2} , we get the following results that these inequalities  are some generalizations of the Audenaert inequality \eqref{ader}.
\begin{corollary}\label{main4}
Let $A, B\in\mathcal{M}_n$, $m,n,s,t$ be real numbers such that $m+n=s+t=1$ and $|||\cdot|||$ be an arbitrary unitarily invariant norm. Then
\begin{align*}
\left|\left|\left|AB^*\right|\right|\right|^2\nonumber&\leq
\left|\left|\left|p(A^*A)^{\frac{m}{p}}+ (1-p)(B^*B)^{\frac{s}{1-p}}\right|\right|\right|\,\left|\left|\left|(1-p)(A^*A)^{\frac{n}{1-p}}+ p(B^*B)^{\frac{t}{p}}\right|\right|\right|,
\end{align*}
where $p\in[0,1]$.
\end{corollary}
\begin{corollary}\label{main3}
Let $A, B\in\mathcal{M}_n$, $m,n,s,t$ be real numbers such that $m+n=s+t=2$ and $|||\cdot|||$ be an arbitrary unitarily invariant norm. Then
{\begin{align}\label{man3}
\left|\left|\left|AB^*\right|\right|\right|^2\nonumber&\leq
\left|\left|\left|p(A^*A)^{m}+ (1-p)(B^*B)^{s}\right|\right|\right|^\frac{1}{2}\left|\left|\left|(1-p)(A^*A)^{m}+ p(B^*B)^{s}\right|\right|\right|^\frac{1}{2}\\&\quad\times\left|\left|\left|p(A^*A)^{n}+ (1-p)(B^*B)^{t}\right|\right|\right|^\frac{1}{2}\left|\left|\left|(1-p)(A^*A)^{n}+p(B^*B)^{t}\right|\right|\right|^\frac{1}{2},
\end{align}}
in which $p\in[0,1]$.
\end{corollary}

\begin{remark}
If we put $n=m=s=t=1$ in inequality \eqref{man3}, then we obtain the Audenaert inequality \eqref{ader}.
Also, if we use inequality \eqref{momo2}, Corollaries \ref{main4} and \ref{main3}, then similar to Corollaries \ref{lolo1} and \ref{lolo2} we get the following inequalities
{\begin{align}\label{kitkit}
\|AXB^*\|_2^2\leq
\|p(A^*A)^{\frac{m}{p}}X+ (1-p)X(B^*B)^{\frac{s}{1-p}}\|_2\|(1-p)(A^*A)^{\frac{n}{1-p}}X+ pX(B^*B)^{\frac{t}{p}}\|_2,
\end{align}}
where $A, B\in\mathcal{M}_n$, $m,n,s,t$ are real numbers  such that $m+n=s+t=1$, $p\in[0,1]$ and
 {\begin{align*}
\|AXB^*\|_2^2\nonumber&\leq
\|p(A^*A)^{m}X+ (1-p)X(B^*B)^{s}\|_2^\frac{1}{2}\|(1-p)(A^*A)^{m}X+ pX(B^*B)^{s}\|_2^\frac{1}{2}\\&\quad\times\|p(A^*A)^{n}X+ (1-p)X(B^*B)^{t}\|_2^\frac{1}{2}\|(1-p)(A^*A)^{n}X+pX(B^*B)^{t}\|_2^\frac{1}{2}
\end{align*}}
for $A, B\in\mathcal{M}_n$, real numbers $m,n,s,t$  such that $m+n=s+t=2$ and $p\in[0,1]$.  These inequalities are generalizations of \eqref{newader} for the Hilbert-Schmidt norms.
\end{remark}
In the following theorem, we show a refinement of inequality \eqref{kitkit}  for the Hilbert-Schmidt norm.
\begin{theorem}\label{main3e}
Let $A, B, X\in\mathcal{M}_n$. Then
{\footnotesize\begin{eqnarray*}
\left\|AXB^*\right\|_2^2&\leq&
\left(\left\|p (A^*A)^\frac{m}{p} X+(1-p) X(B^*B)^\frac{s}{1-p}\right\|_2^2\,-r^2\left\|(A^*A)^\frac{m}{p}X-X(B^*B)^\frac{s}{1-p}\right\|_2^2\right)\nonumber\\&&
\times\left(\left\|(1-p) (A^*A)^\frac{n}{1-p} X+p X(B^*B)^\frac{t}{p}\right\|_2^2\,-r^2\left\|(A^*A)^\frac{n}{1-p}X-X(B^*B)^\frac{t}{p}\right\|_2^2\right),
\end{eqnarray*}}
in which $m,n,s,t$ are real numbers such that $m+n=s+t=1$, $p\in[0,1]$ and $r=\min\{p,1-p\}$.
\end{theorem}
\begin{proof}
Using inequality \eqref{4}, we obtain
{\footnotesize \begin{eqnarray*}
\left\|AXB^*\right\|_2^2&\leq&\left\|(A^*A)^mX(B^*B)^s\right\|_2\left\|(A^*A)^nX(B^*B)^t\right\|_2\\&=&\left\|\left((A^*A)^\frac{m}{p}\right)^pX\left((B^*B)^\frac{s}{1-p}\right)^{1-p}\right\|_2\left\|\left((A^*A)^\frac{n}{1-p}\right)^{1-p}X\left((B^*B)^\frac{t}{p}\right)^{p}\right\|_2\\&\leq&
\left(\left\|p (A^*A)^\frac{m}{p} X+(1-p) X(B^*B)^\frac{s}{1-p}\right\|_2^2\,-r^2\left\|(A^*A)^\frac{m}{p}X-X(B^*B)^\frac{s}{1-p}\right\|_2^2\right)\\&&
\times\left(\left\|(1-p) (A^*A)^\frac{n}{1-p} X+p X(B^*B)^\frac{t}{p}\right\|_2^2\,-r^2\left\|(A^*A)^\frac{n}{1-p}X-X(B^*B)^\frac{t}{p}\right\|_2^2\right),
\end{eqnarray*}}
where $p\in[0,1]$ and $r=\min\{p,1-p\}$, and the proof is complete.
\end{proof}
Theorem \ref{main3e} includes a special case as follows.
\begin{corollary}\cite[Theorem 2.5]{kit1}
Let $A, B, X\in\mathcal{M}_n$. Then
{\footnotesize\begin{eqnarray*}
\left\|AXB^*\right\|_2^2&\leq&
\left(\left\|p A^*A X+(1-p) XB^*B\right\|_2^2\,-r^2\left\|A^*AX-XB^*B\right\|_2^2\right)\\&&
\times\left(\left\|(1-p) (A^*A) X+p XB^*B\right\|_2^2\,-r^2\left\|A^*AX-XB^*B\right\|_2^2\right),
\end{eqnarray*}}
where  $p\in[0,1]$ and $r=\min\{p,1-p\}$.
\end{corollary}
\begin{proof}
For $p\in[0,1]$, if we put  $m=t=p$ and $n=s=1-p$ in Theorem \ref{main3e}, then we get the desired result.
\end{proof}
The next result is a refinement of inequality \eqref{kiteq2}.
\begin{theorem}\label{m7}
Let $A, B, X\in\mathcal{M}_n({\mathbb C})$ and let $p\in (0, 1)$. Then\\
\\
$(\rm i)$ for $0<p\leq\frac{1}{2}$,
{\footnotesize\begin{align}\label{m2}
&\|AXB^{*}\|_{2}^{2}\leq\nonumber\\&
\Big(\|pA^{*}AX+(1-p)XB^{*}B\|_{2}^{2}-r_{0}\|(A^{*}A)^{\frac{1}{2}}X(B^{*}B)^{\frac{1}{2}}-A^{*}AX\|_{2}^{2}
-(1-p)^{2}\|A^{*}AX-XB^{*}B\|_{2}^{2} \Big)^{\frac{1}{2}}\nonumber\\&
\times \Big(\|(1-p)A^{*}AX+pXB^{*}B\|_{2}^{2}-r_{0}\|(A^{*}A)^{\frac{1}{2}}X(B^{*}B)^{\frac{1}{2}}-A^{*}AX\|_{2}^{2}
-p^{2}\|A^{*}AX-XB^{*}B\|_{2}^{2} \Big)^{\frac{1}{2}}.
\end{align} }
$(\rm ii)$ for $\frac{1}{2}<p<1$,
{\footnotesize\begin{align}\label{m3}
&\|AXB^{*}\|_{2}^{2}\leq\nonumber\\&\Big(\|pA^{*}AX+(1-p)XB^{*}B\|_{2}^{2}-r_{0}\|(A^{*}A)^{\frac{1}{2}}X(B^{*}B)^{\frac{1}{2}}-X(B^{*}B)\|_{2}^{2}
-(1-p)^{2}\|A^{*}AX-XB^{*}B\|_{2}^{2} \Big)^{\frac{1}{2}}\nonumber\\&
\times \Big(\|(1-p)A^{*}AX+pXB^{*}B\|_{2}^{2}-r_{0}\|(A^{*}A)^{\frac{1}{2}}X(B^{*}B)^{\frac{1}{2}}-X(B^{*}B)\|_{2}^{2}
-p^{2}\|A^{*}AX-XB^{*}B\|_{2}^{2} \Big)^{\frac{1}{2}},
\end{align}}
where $r=\min\{p, 1-p\}$ and $r_{0}=\min\{2r, 1-2r\}$.
\end{theorem}
\begin{proof}
 The proof of inequality \eqref{m3} is similar to that of inequality \eqref{m2}. Thus, we only need to prove the inequality \eqref{m2}.\\
 If $0<p\leq\frac{1}{2}$, then we replace $A$ and $B$ by $A^{*}A$ and $B^{*}B$ in inequality \eqref{mm}, respectively, we have
 {\footnotesize\begin{align}\label{m4}
 &\|(A^{*}A)^{p}X(B^{*}B)^{1-p}\|_{2}\leq\nonumber\\&
 \Big(\|pA^{*}AX+(1-p)XB^{*}B\|_{2}^{2}-r_{0}\|(A^{*}A)^{\frac{1}{2}}X(B^{*}B)^{\frac{1}{2}}-A^{*}AX\|_{2}^{2}
-(1-p)^{2}\|A^{*}AX-XB^{*}B\|_{2}^{2} \Big)^{\frac{1}{2}}.
\end{align}}
Interchanging the roles of $p$ and $1-p$ in the inequality \eqref{m4}, we get
 {\footnotesize\begin{align}\label{m5}
 &\|(A^{*}A)^{1-p}X(B^{*}B)^{p}\|_{2}\leq\nonumber\\&
 \Big(\|(1-p)A^{*}AX+pXB^{*}B\|_{2}^{2}-r_{0}\|(A^{*}A)^{\frac{1}{2}}X(B^{*}B)^{\frac{1}{2}}-A^{*}AX\|_{2}^{2}
-p^{2}\|A^{*}AX-XB^{*}B\|_{2}^{2} \Big)^{\frac{1}{2}}.
\end{align}}
Applying  inequalities \eqref{4}, \eqref{m4} and \eqref{m5}, we get the desired result.
\end{proof}
\begin{corollary}
Let $A, B\in\mathcal{M}_n({\mathbb C})$ and $p\in (0, 1)$. Then\\
\\
$(\rm i)$ for $0<p\leq\frac{1}{2}$,
{\footnotesize\begin{align*}
&\|AB^{*}\|_{2}^{2}\leq\\&
\Big(\|pA^{*}A+(1-p)B^{*}B\|_{2}^{2}-r_{0}\|(A^{*}A)^{\frac{1}{2}}(B^{*}B)^{\frac{1}{2}}-A^{*}A\|_{2}^{2}
-(1-p)^{2}\|A^{*}A-B^{*}B\|_{2}^{2} \Big)^{\frac{1}{2}}\\&
\times \Big(\|(1-p)A^{*}A+pB^{*}B\|_{2}^{2}-r_{0}\|(A^{*}A)^{\frac{1}{2}}(B^{*}B)^{\frac{1}{2}}-A^{*}A\|_{2}^{2}
-p^{2}\|A^{*}A-B^{*}B\|_{2}^{2} \Big)^{\frac{1}{2}}.
\end{align*}}
$(\rm ii)$ for $\frac{1}{2}<p<1$,
{\footnotesize\begin{align*}
&\|AB^{*}\|_{2}^{2}\leq\\&
\Big(\|pA^{*}A+(1-p)B^{*}B\|_{2}^{2}-r_{0}\|(A^{*}A)^{\frac{1}{2}}(B^{*}B)^{\frac{1}{2}}-(B^{*}B)\|_{2}^{2}
-(1-p)^{2}\|A^{*}A-B^{*}B\|_{2}^{2} \Big)^{\frac{1}{2}}\\&
\times \Big(\|(1-p)A^{*}A+pB^{*}B\|_{2}^{2}-r_{0}\|(A^{*}A)^{\frac{1}{2}}(B^{*}B)^{\frac{1}{2}}-(B^{*}B)\|_{2}^{2}
-p^{2}\|A^{*}A-B^{*}B\|_{2}^{2} \Big)^{\frac{1}{2}},
\end{align*}}
where $r=\min\{p, 1-p\}$ and $r_{0}=\min\{2r, 1-2r\}$.
\end{corollary}
We would like obtain upper bound for $|||AXB^{*}|||$, for every unitary invariant norm.\\
The following lemma has been shown in \cite{sab}, and considered as a refined matrix
Young's inequality for every unitary invariant norm.
\begin{lemma}
Let $A, B, X\in\mathcal{M}_n$ such that $A, B$ are positive semidefinite. Then for $0\leq p\leq1$, we have
\begin{align}\label{m8}
|||A^pXB^{1-p}|||^2+r_{0}(|||AX|||-|||XB|||)^2\leq(p|||AX|||+(1-p)|||XB|||)^2,
\end{align}
where $r_{0}=\min\{p,1-p\}$.
\end{lemma}
\begin{proposition}
Let $A, B, X\in\mathcal{M}_n$. Then
\begin{align*}
|||AXB^*||||^{2}&\leq\Big((p|||A^*AX|||+(1-p)|||XB^*B|||)^2-r_{0}^{2}(|||A^*AX|||-|||XB^*B|||)^{2}\Big)^\frac{1}{2}\\&
\times\Big(((1-p)|||A^*AX|||+p|||XB^*B|||)^2-r_{0}^{2}(|||A^*AX|||-|||XB^*B|||)^{2}\Big)^\frac{1}{2},
\end{align*}
where $p\in [0,1]$ and $r_0=\min\{p,1-p\}$.
\end{proposition}
\begin{proof}
In inequality \eqref{m8}, we put $A=A^*A$ and $B=B^*B$, we get
{\footnotesize\begin{align}\label{m9}
|||(A^*A)^pX(B^*B)^{1-p}|||\leq\Big((p|||A^*AX|||+(1-p)|||XB^*B|||)^2-r_0^2(|||A^*AX|||-|||XB^*B|||)^2\Big)^{\frac{1}{2}}.
\end{align}}
Interchanging $p$ with $1-p$ in inequality \eqref{m9}, we get
{\footnotesize\begin{align}\label{m10}
|||(A^*A)^{1-p}X(B^*B)^p|||\leq\Big(((1-p)|||A^*AX|||+p|||XB^*B|||)^2-r_0^2(|||A^*AX|||-|||XB^*B|||)^2\Big)^{\frac{1}{2}}.
\end{align}}
Now by inequalities \eqref{4}, \eqref{m9} and \eqref{m10} we get the desired inequality.
\end{proof}
\bibliographystyle{amsplain}

\end{document}